\documentclass[11pt]{article}
\usepackage{amssymb,amsmath,amsthm,hyperref,graphicx,color,array,fullpage,mathrsfs}
\usepackage[notref,notcite,color,final]{showkeys}
\usepackage{amscd}
\usepackage{makeidx}
\usepackage[OT2,T1]{fontenc}
\DeclareSymbolFont{cyrletters}{OT2}{wncyr}{m}{n}
\DeclareMathSymbol{\Sha}{\mathalpha}{cyrletters}{"58}
\makeindex

\newtheorem{theorem}{Theorem}[section]
\newtheorem{lemma}[theorem]{Lemma}
\newtheorem{corollary}[theorem]{Corollary}

\newtheorem{definition}[theorem]{Definition}

\newtheorem{remark}[theorem]{Remark}
\newtheorem{proposition}[theorem]{Proposition}

\begin{document}
\title{\textsc{Heegner Point Kolyvagin System and Iwasawa Main Conjecture }}

\author{\textsc{Xin Wan}}
\date{}

\maketitle
\begin{abstract}
In this paper we prove an anticyclotomic Iwasawa main conjecture proposed by Perrin-Riou for Heegner points when the global sign is $-1$, using a recent work of the author on one divisibility of Iwasawa-Greenberg main conjecture for Rankin-Selberg $p$-adic $L$-functions. As a byproduct we also prove the equality for the above mentioned main conjecture under some local conditions, and an improvement of C.Skinner's result on a converse of Gross-Zagier-Kolyvagin theorem.
\end{abstract}
\section{Introduction}
Let $p\geq 5$ be an odd prime. Let $E$ be an elliptic curve over $\mathbb{Q}$ with square free conductor $N$. Let $T$ be the $p$-adic Tate-module of $E$ and $V=T\otimes_{\mathbb{Z}_p}\mathbb{Q}_p$. Suppose $E$ has good ordinary reduction at $p$. Let $\mathcal{K}$ be a quadratic imaginary field where $p$ and $2$ are split (that $2$ being split is put in \cite{WAN} due to lack of knowledge of Howe duality in characteristic $2$. Seems recent preprint of Wee Teck Gan solved this issue). Let $G_\mathcal{K}$ and $G_\mathbb{Q}$ be the absolute Galois group of $\mathcal{K}$ and $\mathbb{Q}$. Suppose that the representation ${T}|_{G_\mathcal{K}}$ contains $\mathrm{GL}_2(\mathbb{Z}_p)$ (This is needed for the Euler system argument). Consider either of the following assumptions:
\begin{itemize}
\item[(1)] there is at least one prime $q|N$ non-split in $\mathcal{K}$ and that the global root number of $E$ over $\mathcal{K}$ is $-1$;
\item[(2)] for each $q|N$ either $q$ is split in $\mathcal{K}$, or $q$ is ramified and $E$ has non-split multiplicative reduction at $q$, and suppose we have at least one such ramified prime. This implies that the root number of $E$ over $\mathcal{K}$ is $-1$.
\end{itemize}
We fix $\iota_p:\mathbb{C}\simeq \mathbb{C}_p$ and let ${v}$ be the prime of $\mathcal{K}$ above $p$ determined by $\iota_p$. Let $\bar{v}$ be its conjugation. It is well known that there is a Galois stable $1$-dimensional subspace $V^+\subset V$ such that the Galois action of $G_p$ on $V^+$ is given by the cyclotomic character twisted by an unramified character. Let $V^-=V/V^+$, $T^+=V^+\cap T$, $T^-=T/T^+$. By Taniyama-Shimura conjecture \cite{Wiles95} we know that $E$ is associated to a weight 2 cuspidal eigenform with the automorphic representation $\pi=\pi_f$.\\

\noindent Let $\mathcal{K}_\infty$ be the unique $\mathbb{Z}_p^2$-extension of $\mathcal{K}$ unramified outside $p$. The complex conjugation $c$ acts on $\Gamma:=\mathrm{Gal}(\mathcal{K}_\infty/\mathcal{K})$. We let $\Gamma^{\pm}$ be the $1$-dimensional $\mathbb{Z}_p$-space on which $c$-acts as $\pm 1$. Let  $\mathcal{K}_\infty^{\pm}$ be the $\mathbb{Z}_p$-extension of $\mathcal{K}$ such that $\mathrm{Gal}(\mathcal{K}_\infty^{\pm}/\mathcal{K})=\Gamma^{\pm}$. Take topological generators $\gamma^{\pm}$ of $\Gamma^{\pm}$. We let $\mathcal{K}_k$ be the unique subextension of $\mathcal{K}_\infty^-$ with $\mathrm{Gal}(\mathcal{K}_k/\mathcal{K})\simeq \mathbb{Z}/p^k\mathbb{Z}$. Define $\Lambda_\mathcal{K}:=\mathbb{Z}_p[[\Gamma_\mathcal{K}]]$, $\Lambda=\Lambda^-=\mathbb{Z}_p[[\Gamma_\mathcal{K}^-]]$. This $\Lambda\simeq \mathbb{Z}_p[[W]]$ by the $\mathbb{Z}_p$-map sending $\gamma^-$ to $W+1$. Let $\Psi_\mathcal{K}$ be the composition $G_\mathcal{K}\twoheadrightarrow \Gamma_\mathcal{K}\hookrightarrow \Lambda_\mathcal{K}^\times$ and $\Psi: G_\mathcal{K}\twoheadrightarrow \Gamma_\mathcal{K}^-\hookrightarrow \Lambda^\times.$ Let $\varepsilon:\mathcal{K}^\times\backslash \mathbb{A}_\mathcal{K}^\times\rightarrow \Lambda_\mathcal{K}^\times$ be the composition of $\Psi_\mathcal{K}$ with the reciprocity map of class field theory (normalized by the geometric Frobenius). Let $\mathbf{T}$ be the $\Lambda$-adic Galois representation $T\otimes_{\mathbb{Z}_p}\Lambda(\Psi)$. Here $\Lambda(\Psi)$ means the $\Lambda$-adic character with Galois action given by $\Psi$. We define an involution $\iota:\Lambda\simeq \Lambda$ to be the $\mathbb{Z}_p$-morphism sending $(1+W)$ to $(1+W)^{-1}$. We define $\Lambda_\mathcal{K}^*=\mathrm{Hom}_{\mathbb{Z}_p}(\Lambda_\mathcal{K},\mathbb{Q}_p/\mathbb{Z}_p)$ and $\Lambda^*=\mathrm{Hom}_{\mathbb{Z}_p}(\Lambda,\mathbb{Q}_p/\mathbb{Z}_p)$. Define $A:=T\otimes_{\mathbb{Z}_p}\mathbb{Q}_p/\mathbb{Z}_p$ and  $\mathbf{A}:=T\otimes\Lambda(-\Psi)\otimes_\Lambda\Lambda^*$ (convention is slightly different from Howard's). \\

\noindent\underline{Selmer Conditions}:\\
\noindent We define some Selmer conditions ($\mathcal{F},\mathcal{F}_v,\mathcal{F}_{\bar{v}}, v, \bar{v}, \mathrm{str})$ for various Galois representations $T$, $\mathbf{T}$, $\mathbf{A}$, etc (we write $T$ for example in the following). Let:
$$(\mathcal{F})_w=\left\{\begin{array}{ll}\mathrm{ker}\{H^1(\mathcal{K}_w,T)\rightarrow H^1(I_w,T)\}, &w\nmid p\\ \mathrm{ker}\{H^1(\mathcal{K}_w,T)\rightarrow H^1(K_w, T^-)\}& w|p\end{array}\right.$$
$$(\mathcal{F}_v)_w=\left\{\begin{array}{ll}\mathcal{F}_w, &w\not=\bar{v}\\0& w=\bar{v}\end{array}\right.$$
$$(\mathcal{F}_{\bar{v}})_w=\left\{\begin{array}{ll}\mathcal{F}_w, &w\not={v}\\0& w={v}\end{array}\right.$$
$$(v)_w=\left\{\begin{array}{ll}\mathcal{F}_w, &w\nmid p\\0& w=\bar{v}\\ H^1(\mathcal{K}_w,-)& w=v\end{array}\right.$$
$$(\bar{v})_w=\left\{\begin{array}{ll}\mathcal{F}_w, &w\nmid p\\0& w={v}\\ H^1(\mathcal{K}_w,-)& w=\bar{v}\end{array}\right.$$
$$(str)_w=\left\{\begin{array}{ll}\mathcal{F}_w, &w\nmid p\\0& w|p\end{array}\right.$$
We define the local Selmer condition $\mathcal{F}'$ by replacing the $w|p$ local conditions for $\mathcal{F}$ by $\ker\{H^1(\mathcal{K}_w,T)\rightarrow H^1(I_w, T^-)\}$. We also define the local Selmer condition $v'$ by replacing the local condition at $\bar{v}$ by $\ker\{H^1(\mathcal{K}_{\bar{v}},T)\rightarrow H^1(I_{\bar{v}}, T^-)\}$. We define $\bar{v}'$ similarly.

We define the corresponding Selmer groups $H^1_\mathcal{F}(\mathcal{K},-)$ to be the inverse image in $H^1(\mathcal{K},-)$ of $\prod_w \mathcal{F}_w$ under the localization map and $X:=H^1_\mathcal{F}(\mathcal{K},\mathbf{A})^*$. We also define Selmer groups for other Selmer conditions in a similar way.
\subsection{Perrin-Riou's Conjecture}
Let $\Phi: X_0(N)\rightarrow E$ be a modular parameterization given by \cite{Wiles95}. Let $\mathcal{K}[n]$ be the ring class field of $\mathcal{K}$ of conductor $n$.
In \cite{Howard} he constructed a Kolyvagin system $\{\kappa^{Hg}_n\}_n$ with $n$ running over a set of square-free products of degree two primes in $\mathcal{K}$ and $\kappa_1^{Hg}\in H^1(\mathcal{K},\mathbf{T})$. We also adopt Howard's notation that $\kappa_1$ being the image of $\kappa_1^{Hg}$ in $H^1(\mathcal{K},T)$. These depend on the choice of $\Phi$. We refer to \cite{Howard} for the definition of a Kolyvagin system. He proved the following theorem:
\begin{theorem}\label{1.1}
$H_\mathcal{F}^1(\mathcal{K},\mathbf{T})$ is a torsion-free rank one $\Lambda$-module. Write $\mathrm{char}$ for characteristic ideal. Then there is a torsion $\Lambda$-module $M$ such that $\mathrm{char}(M)=\mathrm{char}(M)^\iota$ and a pseudo-isomorphism:
$$X\sim\Lambda\oplus M\oplus M.$$
Also, under assumption (1) we have $\mathrm{char}(M)\supseteq\mathrm{char}(H_\mathcal{F}^1(\mathcal{K},\mathbf{T})/\Lambda\kappa_1^{Hg})$.
\end{theorem}

Howard proved it under the assumption that all primes of $N$ are split in $\mathcal{K}$. However we will see in the next section that it is true under our assumptions as well. The first main result of our paper is the following.
\begin{theorem}\label{1.2}
If assumption (2) is true then
$$\mathrm{char}_\Lambda(M)= \mathrm{char}_\Lambda(H_\mathcal{F}^1(\mathcal{K},\mathbf{T})/\Lambda\kappa_1^{Hg}).$$
If assumption (1) is true then the above equality holds as ideals of $\Lambda\otimes_{\mathbb{Z}_p}\mathbb{Q}_p$.
\end{theorem}
The Theorem is proved by comparing different Selmer groups using Poitou-Tate exact sequences. The key ingredient is the two-variable main conjecture for $v$-Selmer groups recently proved by the author \cite{WAN}. We mention that we can not resort to the two-variable main conjecture in \cite{SU} since the global sign is assumed to the $+1$ there. In fact our theorem can be interpreted as a Gross-Zagier theoretic Iwasawa main conjecture: \cite{Howard1} proved a formula relating the right hand side with the quotient of the derivative of a two variable $p$-adic $L$-function and some $\Lambda$-adic regulator $\mathcal{R}$. Thus our theorem simply states that the torsion part of the Selmer group is given by the derivatives of the $p$-adic $L$-function over $\mathcal{R}$. We also mention that this result can be proven under a different assumption (basically assuming some local Galois representations being ramified) by using the idea of W.Zhang $\cite{Zhang}$ combined with results of Howard \cite{Howard2}.
\begin{remark}
In general there should be a Manin constant showing up in the above identity. However by our assumption of square-free conductor and assumption (2) such Manin constant is prime to $p$. There should also be a product of local Tamagawa numbers which is again a $p$-adic unit under assumption (2). In general the Kolyvagin system argument does not give the sharp upper bound. We will treat these gaps in a future joint work with D. Jetchev and C.Skinner.
\end{remark}
The idea in this paper is actually the first case of a much more general strategy for Iwasawa theory that we will carry out: one uses special cycles (as Heegner points in this paper) to build a link among different kinds of main conjectures and thus reduce some more difficult conjectures to a relatively accessible one. We have already seen such strategy to work in other cases: in \cite{WAN1} we again use Heegner points but in the supersingular case to prove that the torsion part of the $\pm$-Selmer group is the square of divisibility of the $\pm$-Heegner points. In \cite{WAN2} we use Beilinson-Flach elements as special cycles to prove Kobayashi's cyclotomic main conjecture for supersingular elliptic curves. Such results were previously out of reach by known techniques.

\subsection{Iwasawa Main Conjecture}
We also prove a theorem about full equality in the $2$-variable Iwasawa main conjecture
as a byproduct. First we recall here the main result of \cite{WAN}. As in \emph{loc.cit} there is a $p$-adic $L$-function $\mathcal{L}_{f,\mathcal{K}}\in \mathrm{Frac}\Lambda_\mathcal{K}$, such that for a Zariski dense set of arithmetic points $\phi\in\mathrm{Spec}\Lambda_\mathcal{K}$ such that $\phi\circ \boldsymbol{\varepsilon}$ is the $p$-adic avatar of a Hecke character $\xi_\phi$ of $\mathcal{K}^\times\backslash \mathbf{A}_\mathcal{K}^\times$ of infinite type $(\frac{\kappa}{2},-\frac{\kappa}{2})$ for some $\kappa\geq 6$, of conductor $(p^t,p^t)$ ($t>0$) at $p$, then:
$$\phi(\mathcal{L}_{{f},\mathcal{K}})=\frac{p^{(\kappa-3)t}\xi_{1,p}^2\chi_{1,p}^{-1}
\chi_{2,p}^{-1}(p^{-t})\mathfrak{g}(\xi_{1,p}\chi_{1,p}^{-1})
\mathfrak{g}(\xi_{1,p}\chi_{2,p}^{-1})L(\mathcal{K},f,\xi_\phi,\frac{\kappa}{2})(\kappa-1)!(\kappa-2)!}{(2\pi i)^{2\kappa-1}\Omega_\infty^{2\kappa}}.$$
Here $\mathfrak{g}$ is the Gauss sum and $\chi_{1,p}, \chi_{2,p}$ are characters such that $\pi(\chi_{1,p},\chi_{2,p})\simeq \pi_{f,p}$.
\noindent This $p$-adic $L$-function is obtained by putting back the Euler factors to a $\Sigma$-primitive $p$-adic $L$-function in $\Lambda$. So a priory we do not know if this $p$-adic $L$-function is in $\Lambda_\mathcal{K}$. However we will prove it is the case under the assumptions (2). (This should be true in general by the forthcoming work of Eischen-Harris-Li-Skinner which made careful choices at all ramified primes. However we do not need this).\\

\noindent On the arithmetic side the local Selmer condition in \cite{WAN} is the $v'$-one defined here.
We make some remarks about the slight difference between the Selmer condition $v$ and $v'$ above, We have the Hochschild-Serre spectral sequence:
$$0\rightarrow H^1(G_{\bar{v}}/I_{\bar{v}},\mathbf{A}^{I_{\bar{v}}})\rightarrow H^1(\mathcal{K}_{\bar{v}},\mathbf{A})\rightarrow H^1(I_{\bar{v}},\mathbf{A})^{G_{\bar{v}}}$$
we have
$0\rightarrow H_{\bar{v}}^1(\mathcal{K}_{\bar{v}},\mathbf{A})\rightarrow H_{\bar{v}'}^1(\mathcal{K}_{\bar{v}},\mathbf{A})\rightarrow H^1(G_{\bar{v}}/I_{\bar{v}},\mathbf{A}^{I_{\bar{v}}})$ where the last term is of finite cardinality. This implies that $\mathrm{ord}_P(H_{\bar{v}}^1(\mathcal{K},\mathbf{A})^*)=\mathrm{ord}_P(H_{\bar{v}}^1(\mathcal{K},\mathbf{A},\bar{v}')^*)$ for all height one primes $P$ of $\Lambda$.\\

\noindent The two variable main conjecture states that:
$$\mathrm{char}_{\Lambda_\mathcal{K}}(X_{f,\mathcal{K},v})=(\mathcal{L}_{f,\mathcal{K}})$$
Where $X_{f,\mathcal{K},v}:=H_v^1(\mathcal{K}, T\otimes \Lambda_\mathcal{K}(\Psi_\mathcal{K})\otimes_{\Lambda_\mathcal{K}}\Lambda_\mathcal{K}^*)^*$.
The following is the part (ii) of the main theorem of \cite{WAN}.
\begin{theorem}
Under assumption (1) we have one containment $$\mathrm{char}_{\Lambda_\mathcal{K}\otimes\mathbb{Q}_p}(X_{f,\mathcal{K},v})\subseteq(\mathcal{L}_{f,\mathcal{K}})$$
as fractional ideals of $\Lambda_\mathcal{K}$.
\end{theorem}
\begin{remark}
In \emph{loc.cit} there is a CM character $\xi$ of $\mathcal{K}^\times\backslash \mathbb{A}^\times_\mathcal{K}$ showing up. In fact the assumption on $\xi$ in that part (ii) there exactly means some specialization of the $p$-adic $L$-functions gives the ``special value $L(\mathcal{K},f,1)$'' but this is not an interpolation point since here the CM form is a weight one Eisenstein series which has weight less than $f$.
\end{remark}
Our second main theorem is the following
\begin{theorem}
Then under assumption (2) we have $\mathcal{L}_{f,\mathcal{K}}\in\Lambda_\mathcal{K}$. Moreover: $X_{f,\mathcal{K}}$ is $\Lambda_\mathcal{K}$-torsion and
$$\mathrm{char}_{\Lambda_\mathcal{K}}(X_{f,\mathcal{K}})=(\mathcal{L}_{f,\mathcal{K}}).$$
\end{theorem}
\subsection{Converse of Gross-Zagier-Kolyvagin Theorem}
Finally we prove the following result which is a stronger form of the result of C.Skinner in \cite{Skinner}, which has been used as an important ingredient to prove that a majority of elliptic curves satisfy the BSD conjecture.
\begin{theorem}
Let $E/\mathbb{Q}$ be an elliptic curve with square-free conductor $N$ and good ordinary reduction at $p$. Let $\mathcal{K}$ be an imaginary quadratic field such that $p$ splits and $E[p]|_{G_\mathcal{K}}$ is irreducible. Suppose the assumption (1) is true. If the Selmer group $H_{\mathcal{F}}^1(\mathcal{K}, E[p^\infty])$ has corank one. Then the Heegner point $\kappa_1$ is not torsion and thus the vanishing order of $L_\mathcal{K}(f,1)$ is exactly one.
\end{theorem}
\begin{remark}
Here by arguing Iwasawa theoretically we can remove the assumption put in \cite{Skinner} that certain localization map at $p$ is injective. In fact such injectivity has been replaced by a $\Lambda$-adic one, which is automatically true by the non-triviality of the family of Heegner points.
\end{remark}
One may also deduce some results over $\mathbb{Q}$ by choosing the $\mathcal{K}$ properly as in \cite{Skinner}. We omit the details.
\section{Known Results}
\subsection{Ben Howard's Result}
In \cite{Howard} Howard constructed a Kolyvagin system $\kappa_1^{Hg}\in H^1(\mathcal{K},\mathbf{T})$, and proved the following theorem under a different assumption.
\begin{theorem}\label{How}
$H_\mathcal{F}^1(\mathcal{K},\mathbf{T})$ is a torsion free rank one $\Lambda$-module. There is a torsion $\Lambda$-module $M$ such that $\mathrm{char}(M)=\mathrm{char}(M)^\iota$ and a pseudo-isomorphism:
$$X\sim\Lambda\oplus M\oplus M.$$
Also, $\mathrm{char}(M)|\mathrm{char}(H_\mathcal{F}^1(\mathcal{K},\mathbf{T})/\Lambda\kappa_1^{Hg})$.
\end{theorem}
\begin{proof}
This is essentially \cite[theorem 2.2.10]{Howard} except that we did not assume that all primes dividing $N$ are split in $\mathcal{K}$. However our local condition is enough to construct the Heegner points. We refer to \cite{CV} for the construction of the Heegner points $\kappa_1^{Hg}$ in the more general case. In \cite{Howard} Howard used a result of Cornut \cite{Cornut} that the image of certain Heegner point under the trace map of the Galois group $G_0:=\mathrm{Gal}(K[p^\infty]/H_\infty)$ on the elliptic curve $E$ is non-torsion (here $K[p^\infty]$ is the ring class field of $K$ with conductor $p^\infty$ and $H_\infty$ is the unique anti-cyclotomic $\mathbb{Z}_p$-extension of $\mathcal{K}$). Cornut assumed that all prime factors of $N$ splits in $K$. The general case is treated in \cite{CV} (see Theorem 1.5 of \emph{loc.cit} and the discussion right after it).\\

\noindent Now Howard's proof works throughout. The only difference is in lemma 2.3.4 of \emph{loc.cit} case (ii) we need to take care of non-split primes $v|N$. But then any prime of $\mathcal{K}[n]$ above $v$ splits completely in $\mathcal{K}_k[n]$, which is the composition of $\mathcal{K}_k$ and $\mathcal{K}[n]$ where $\mathcal{K}[n]$ is the ring class field of conductor $n$ for $n$ a square-free product of inert primes. So the fact that $\kappa_1^{Hg}$ is in the unramified class follows from that the inertial group of $v$ is the same as that for any prime of $\mathcal{K}_\infty[n]$ above $v$.
\end{proof}

\subsection{Castella's Formula}
Now we recall the result of \cite{Castella} which generalizes a formula of \cite{BDP}. There is a big logarithm map $\log_{\omega_E}: H^1_\mathcal{F}(K_v,\mathbf{T})\hookrightarrow \Lambda$ with finite cokernel (\cite{Castella}, note that by the construction there the image is in $\Lambda$ and contains the maximal ideal of $\Lambda$). Moreover there is a quasi-isomorphism $H^1(\mathcal{K}_v,\mathbf{T})/H_\mathcal{F}^1(\mathcal{K}_v,\mathbf{T})\rightarrow \Lambda$ with finite kernel and cokernel. We write $\mathcal{L}^{BDP}_{f,\mathcal{K}}$ for the $p$-adic $L$-function of Bertolini-Darmon-Prasanna \cite{BDP}, which is the specialization of $\mathcal{L}_{f,\mathcal{K}}$ to $\gamma^+\rightarrow 1$ by comparing their interpolation formulas.
\begin{definition}
Let $P$ be a height one prime of $\Lambda$. We consider the map:
$H_\mathcal{F}^1(\mathcal{K},\mathbf{T})_P\rightarrow H^1_{\mathcal{F}_v}(\mathcal{K}_v,\mathbf{T})_P$ and $H_\mathcal{F}^1(\mathcal{K},\mathbf{T})_P\rightarrow H^1_{\mathcal{F}_{\bar{v}}}(\mathcal{K}_{\bar{v}},\mathbf{T})_P$. These are maps of free $\Lambda_P$-modules of rank $1$. Moreover we are going to know that these maps are non-zero. We define the numbers $f_{v,P}$ and $f_{\bar{v},P}$ to be the orders of $P$ of the cokernels of the corresponding maps above.
\end{definition}
The following proposition is proved in \cite{Castella}.
\begin{proposition}\label{2.2}
Under assumption (2) we have
$$\mathcal{L}_{f,\mathcal{K}}^{BDP}=\log_{\omega_E}^2(\kappa_1^{Hg}).$$
Under assumption (1) the above is true up to multiplying by some powers of $p$.
\end{proposition}
In fact in \cite{Castella} the result is not stated in such generality since he only uses Heegner points on modular curves (see assumption (iii) of his Theorem A). However if we use Heegner points on general indefinite Shimura curves the proof goes in the completely same way. Some additional assumptions are put in \emph{loc.cit} due to the fact that they are working with Hida families. They are not necessary if we are only interested in a single weight two form with trivial character (\cite{Personal}). There should also be a modular degree factor showing up. Under assumption (2) this degree is co-prime to $p$ by the square-free conductor assumption and that the Heegner points come from the modular curve. But under assumption (1) the formula is only true up to a scalar.
\begin{corollary}\label{Castella}
For any height one prime $P$ of $\Lambda$, we have under assumption (2)
$$\mathrm{ord}_P(\mathcal{L}_{p,f,\mathcal{K}}^{BDP})=2f_{v,P}+2\mathrm{ord}_P(H_\mathcal{F}^1(\mathcal{K},
\mathbf{T})/\kappa_1^{Hg}).$$
Under assumption (1) the above is true for all $P\not=(p)$.
\end{corollary}
\begin{proof}
We consider the map:
$$0\rightarrow H_\mathcal{F}^1(\mathcal{K},\mathbf{T})_P\rightarrow H_\mathcal{F}^1(\mathcal{K}_v,\mathbf{T})_P\rightarrow \Lambda_P$$
of $\Lambda_P$-modules. Here the last is the $\log_{\omega_E}$ map. The corollary is clear.
\end{proof}
\section{Proof of the Main Results}
\subsection{Tate Local Duality}
Let $K$ be a finite extension of $\mathbb{Q}_p$. Let $V^*:=\mathrm{Hom}_{\mathbb{Q}_p}(V,\mathbb{Q}_p(1))$. The pairing $V\times V^*\rightarrow \mathbb{Q}_p(1)$ gives a perfect pairing (see \cite[1.4]{Rubin}):
$$H^1(K,V)\otimes H^1(K,V^*)\rightarrow H^2(K,\mathbb{Q}_p(1))\simeq \mathbb{Q}_p.$$
We also have a perfect pairing:
$$H^1(K,T)\times H^1(K,{T}^*\otimes_{\mathbb{Z}_p}\mathbb{Q}_p/\mathbb{Z}_p)\rightarrow H^2(K,\mathbb{Q}_p/\mathbb{Z}_p(1))\simeq \mathbb{Q}_p/\mathbb{Z}_p.$$
We usually write $\mathcal{F}^*$ for the dual Selmer conditions by requiring $\mathcal{F}_w^*$ is the orthogonal complement of $\mathcal{F}_w$ under the local Tate pairing.
\subsection{Control of Selmer Groups}\label{control}
We first recall some notion of Greenberg \cite{Greenberg} about control theorems for Selmer groups. Let $F$ be an extension of $\mathcal{K}$ ($\mathcal{K}$ or $\mathcal{K}_\infty^-$ in application) and $F_\infty$ is an extension of $F$ such that the Galois group is isomorphic to $\Gamma=\mathbb{Z}_p^d$ for some $d$. We have the following diagram:

\[  \begin{CD}
0@>>> H_v^1(F, E[p^\infty])@>>>H^1(F, E[p^\infty])@>>>\mathcal{G}_E(F)@>>>0\\
@.     @VsVV                     @VhVV                    @VgVV           @.\\
0@>>> H_v^1(F_\infty, E[p^\infty])^\Gamma @>>>H^1(F_\infty, E[p^\infty])^\Gamma @>>>\mathcal{G}_E(F_\infty)^\Gamma @.
\end{CD}\]
Where
$$\mathcal{G}_E(M)=\mathrm{Im}(H^1(M, E[p^\infty])\rightarrow \mathcal{P}_E(M))$$
 and

$$\mathcal{P}_E(M)=\prod_{\eta\nmid p}H^1(M_\eta, E[p^\infty])/H_\mathcal{F}^1(M_\eta,E[p^\infty])\prod_{\eta|p}H^1_v(G_\eta, E[p^\infty]).$$
For any prime $w$ of $F$ let $r_w$ be the map
$$H^1(F_w,E[p^\infty])/H_v^1(F_w, E[p^\infty])\rightarrow H^1(F_{\infty,w},E[p^\infty])/H_v^1(F_{\infty,w}, E[p^\infty]).$$
By the snake lemma we have
$$0\rightarrow \mathrm{ker}s\rightarrow \mathrm{ker}h\rightarrow \mathrm{ker}g\rightarrow \mathrm{coker}s\rightarrow \mathrm{coker}h.$$
We relate the $2$-variable and $1$-variable Selmer groups of $f$ with local condition ``$v$'' by specializing the  cyclotomic variable. We let $I$ be the prime ideal of $\mathcal{O}_L[[\Gamma_\mathcal{K}]]$ generated by $(\gamma^+-1)$. Let $X_{f,\mathcal{K},v}^{anti}:=H_v^1(\mathcal{K}, T\otimes \Lambda(\Psi)\otimes_\Lambda\Lambda^*)^*$.
\begin{proposition}
Under assumption (1) there is an isomorphism
$X_{f,\mathcal{K},v}/I\otimes L\simeq X_{f,\mathcal{K},v}^{anti}\otimes L$ of $\mathcal{O}_L[[\Gamma_\mathcal{K}^-]]\otimes L$-modules. Under assumption (2) the above equality is true as $\Lambda$-modules.
\end{proposition}
\begin{proof}
We consider the above discussion with $F=\mathcal{K}_\infty^-$ and $F_\infty=\mathcal{K}_\infty$. We need to study the kernel and cokernel of $s$. First note that $\mathrm{ker}h=0$ and $\mathrm{coker}h=0$ (see \cite[lemma 3.2]{Greenberg}). So $\mathrm{ker}s=0$. We need to study $\mathrm{ker}(g)$. For any prime $w$ of $\mathcal{K}$ split in $\mathcal{K}/\mathbb{Q}$ (allowed to divide $p$) and primes $w^-$ of $\mathcal{K}_\infty^-$ and $w_\infty$ of $\mathcal{K}_\infty$ above it, we have $\mathcal{K}_{\infty,w^-}^-=\mathcal{K}_{\infty,w_\infty}$, so the $\mathrm{ker}r_{w^-}=0$.

At non-split primes $w$ and $w^-$ and $w_\infty$ as above, we know $w$ is completely split in $\mathcal{K}_\infty^-$ and so $\mathcal{K}_{\infty,v^-}^-=\mathcal{K}_v$. Since $\mathbb{Q}_\infty\subset \mathcal{K}_\infty$, so by the \cite[page 74]{Greenberg}, $\mathrm{ker}(r_{v^-})\sim c_v^{(p)}$ where $c_v^{(p)}$ is the $p$-part of the Tamagawa number of $E$ at $v$. Under (2) by our assumption about non-split multiplicative reduction this is a $p$-adic unit. Under (1) such kernel is killed when tensoring with $\mathbb{Q}_p$. These altogether shows that the two-variable main conjecture implies the one variable anticyclotomic main conjecture.
\end{proof}
\begin{corollary}\label{3.8}
We have under assumption (2)
$$(\mathcal{L}_{f,\mathcal{K}}^{BDP})\supseteq \mathrm{char}_\Lambda(X_{f,\mathcal{K},v}^{anti}).$$
Under assumption (1) the above is true as ideals of $\Lambda\otimes_{\mathbb{Z}_p}\mathbb{Q}_p$.
\end{corollary}
\begin{proof}
We already have the result after inverting $p$ by combining the above proposition with the main theorem of \cite{WAN}. The powers of $p$ can be taken care of by note that by \cite{Hsieh} $L^{BDP}_{f,\mathcal{K}}$ has $\mu$-invariant 0 under assumption (2).
\end{proof}
\subsection{Galois Cohomology Computations}
We say a $\Lambda$-module $M$ is of rank $n$ if $M\otimes_\Lambda F_{\Lambda}$ is dimension $n$ over $F_\Lambda$ ($F_\Lambda$ is the fraction field of $\Lambda$). Let $\Sigma$ be a finite set of primes containing all primes such that $E$ or $\mathcal{K}$ is ramified. As in \cite{MR} we define a set of height one primes of $\Lambda$
$$\Sigma_\Lambda:=\{P:\sharp H^2(\mathcal{K}_\Sigma/\mathcal{K},\mathbf{T})[P]=\infty\}\cup \{P:\sharp H^2(\mathcal{K}_v,\mathbf{T})[P]=\infty\}\cup \{P:\sharp H^2(\mathcal{K}_{\bar{v}},\mathbf{T})[P]=\infty\}\cup\{(p)\}.$$
This is a finite set by \cite[lemma 5.3.4]{MR}.
\begin{proposition}(Poitou-Tate exact sequence)
Let $\mathcal{F}\subseteq \mathcal{G}$ be two Selmer conditions. We have the following long exact sequence:
$$0\rightarrow H_\mathcal{F}^1(\mathcal{K},\mathbf{T})\rightarrow H_\mathcal{G}^1(\mathcal{K},\mathbf{T})\rightarrow H_{\mathcal{G}}^1(\mathcal{K}_v,\mathbf{T})/H_\mathcal{F}^1(\mathcal{K}_v,\mathbf{T})\rightarrow H_{\mathcal{F}^*}^1(\mathcal{K},\mathbf{A})^*\rightarrow H_{\mathcal{G}^*}^1(\mathcal{K},\mathbf{A})^*\rightarrow 0.$$
\end{proposition}
\begin{proof}
It follows from \cite[Theorem 1.7.3]{Rubin} in a similar way as corollary 1.7.5 in \emph{loc.cit}.
\end{proof}
\begin{lemma}
$$H_{str}^1(\mathcal{K},\mathbf{T})=0.$$
\end{lemma}
\begin{proof}
We have
$$0\rightarrow H_{str}^1(\mathcal{K},\mathbf{T})\rightarrow H_\mathcal{F}^1(\mathcal{K},\mathbf{T})\rightarrow \oplus_w H^1_\mathcal{F}(\mathcal{K}_w,\mathbf{T})/H_{str}^1(\mathcal{K}_w,\mathbf{T}).$$
We tensor it with $F_\Lambda$. By proposition \ref{2.2}, the image of $\kappa_1^{Hg}\in H^1_\mathcal{F}(\mathcal{K}_v,\mathbf{T})$ in $H^1_\mathcal{F}(\mathcal{K}_v,\mathbf{T})$ is the $p$-adic $L$-function of Bertolini-Darmon-Prasanna. This is non-zero under assumption (2) by the result of \cite{Hsieh}. More generally under assumption (1) recall that in the proof of Theorem \ref{How} we have by \cite{CV} for some character $\chi$ of $\Gamma^-$, the specialization of $\kappa_1^{Hg}$ to $\chi$ is non-torsion. On the other hand note that the map $E(\mathcal{K}_n)\rightarrow E(\mathcal{K}_{n,v})$ is injective. So the image of $\kappa_1^{Hg}$ in $H^1_\mathcal{F}(\mathcal{K}_v,\mathbf{T})$ is still non-zero. Since $H_\mathcal{F}^1(\mathcal{K},\mathbf{T})$ is rank one, $H_{str}^1(\mathcal{K},\mathbf{T})$ must be of rank $0$. We know that $H^1(\mathcal{K},\mathbf{T})$ is torsion free as remarked in \cite[lemma 2.2.9]{Howard}. So $H_{str}^1(\mathcal{K},\mathbf{T})$ is $0$.
\end{proof}
It is in this lemma that we used a $\Lambda$-adic version injectivity of the localization map at $p$ which replaced similar assumption in \cite{Skinner}. The advantage of our argument is that such injectivity is automatic.
\begin{lemma}
$H^1(\mathcal{K},\mathbf{T})$ and $H^1(\mathcal{K},\mathbf{A})^*$ have the same $\Lambda$-rank.
\end{lemma}
\begin{proof}
We need only to prove that for a generic set of height one prime $P$, $H^1(\mathcal{K},\mathbf{T})/P$ and $H^1(\mathcal{K},\mathbf{A})^*/P^\iota$ have the same $\mathbb{Z}_p$-rank. For any height one prime $\mathfrak{P}$ we define $S_\mathfrak{P}$ to be the integral closure of $\Lambda/\mathfrak{P}$ and $\mathbf{T}_\mathfrak{P}$ to be the Galois representation obtained by $\mathbf{T}/\mathfrak{P}\mathbf{T}$ base changed to $S_\mathfrak{P}$. Let $\Phi_\mathfrak{P}$ be the fraction field of $S_\mathfrak{P}$ and $A_\mathfrak{P}$ the base change of $\mathbf{T}_\mathfrak{P}$ to $\Phi_\mathfrak{P}/S_\mathfrak{P}$. Suppose $P\not\in\Sigma_\Lambda$ and $g_P$ be a generator of $P$. From:
$$0\rightarrow \mathbf{T}\rightarrow \mathbf{T}\rightarrow \mathbf{T}/P\rightarrow 0$$
where the second arrow is given by multiplication by $g_P$. We have:
$$H^1(\mathcal{K},\mathbf{T})/P\hookrightarrow H^1(\mathcal{K},\mathbf{T}/P)\rightarrow H^2(\mathcal{K},\mathbf{T})[P].$$
From
$$0\rightarrow \mathbf{A}[P^\iota]\rightarrow \mathbf{A}\rightarrow \mathbf{A}\rightarrow 0$$
we have
$$H^0(\mathcal{K},\mathbf{A})\rightarrow H^1(\mathcal{K},\mathbf{A}[P^\iota])\twoheadrightarrow H^1(\mathcal{K},\mathbf{A})[P^\iota].$$
Note that $H^0(\mathcal{K},\mathbf{A})$ has finite cardinality. On the other hand we have:
$$H^1(\mathcal{K},\mathbf{T}/P)\rightarrow H^1(\mathcal{K},\mathbf{T}_P)$$
$$H^1(\mathcal{K},\mathbf{A}_{P^\iota})\rightarrow H^1(\mathcal{K},\mathbf{A}[P^\iota])$$
both have finite kernel and cokernel since $P\in\Sigma_\Lambda$, by \cite[Lemma 2.2.7]{Howard}. Also $H^1(\mathcal{K},\mathbf{T}_P)$ is the $\mathfrak{m}_P$-adic Tate module of $H^1(\mathcal{K},\mathbf{A}_{P^\iota})$. So the $\mathbb{Z}_p$-rank of $H^1(\mathcal{K},\mathbf{T})/P$ is the $\mathbb{Z}_p$-corank of $H^1(\mathcal{K},\mathbf{A}[P^\iota])$ and we are done.
\end{proof}
\begin{corollary}\label{3.5}
The $H^1(\mathcal{K},\mathbf{T})$ has rank $2$ and $H^1_{str}(\mathcal{K},\mathbf{A})^*$ is $\Lambda$-torsion.
\end{corollary}
\begin{proof}
This follows from the above lemma and the Poitou-Tate long exact sequence taking $\mathcal{F}=str$ and $\mathcal{G}$ to have empty restriction at primes above $p$ and same as $\mathcal{F}$ elsewhere, by noting that $H_{str}^1(\mathcal{K},\mathbf{A})^*$ is a quotient of $H_\mathcal{F}^1(\mathcal{K},\mathbf{A})^*$ and thus has rank not greater than one.
\end{proof}

\begin{lemma}\label{3.4}
We have exact sequence:
$$0\rightarrow H_\mathcal{F}^1(\mathcal{K},\mathbf{T})\rightarrow H_{\mathcal{F}_v^*}^1(\mathcal{K},\mathbf{T})\rightarrow \mathrm{coker}\rightarrow 0$$
where $\mathrm{coker}$ has finite cardinality.
\end{lemma}
\begin{proof}
We first claim that $H_{\mathcal{F}_v^*}^1(\mathcal{K},\mathbf{T})$ is a rank $1$ $\Lambda$-module. This follows from a $\Lambda$-adic analogue of the argument \cite[lemma 2.3.2]{Skinner}. The rank is obviously at least one. By the above lemme $H^1(\mathcal{K},\mathbf{T})$ has rank $2$. We consider the image of $H^1(\mathcal{K},\mathbf{T})/H_\mathcal{F}^1(\mathcal{K},\mathbf{T})\hookrightarrow \oplus_v H^1(\mathcal{K}_v,\mathbf{T})/H^1_\mathcal{F}(\mathcal{K}_v,\mathbf{T})$. It is of rank $1$ over $\Lambda$, and is invariant under $\iota\circ c$ ($c$ is the complex conjugation). If $H_{\mathcal{F}_v^*}^1(\mathcal{K},\mathbf{T})$ is rank $2$ then we get that the above image is rank $2$, a contradiction.

Recall we have $H^1(\mathcal{K}_v,\mathbf{T})/H^1_\mathcal{F}(\mathcal{K}_v,\mathbf{T})\hookrightarrow H^1(\mathcal{K}_v,\mathbf{T}'')$ which maps to $\Lambda$ with finite kernel and cokernel. By the Poitou-Tate exact sequence, the $\mathrm{coker}$ is injected to the torsion part of $H^1(\mathcal{K}_v,\mathbf{T})/H^1_\mathcal{F}(\mathcal{K}_v,\mathbf{T})$, which is finite.
\end{proof}
\begin{lemma}\label{3.5}
$H_{\bar{v}}^1(\mathcal{K},\mathbf{T})$ is $0$.
\end{lemma}
\begin{proof}
Since $H^1(\mathcal{K},\mathbf{T})$ is torsion free we only need to show $H_{\bar{v}}^1(\mathcal{K},\mathbf{T})$ is torsion. Again this follows from a $\Lambda$-adic analogue of \cite[lemma 2.3.2]{Skinner} in a completely same way as the above lemma.
\end{proof}

The following Proposition is the key of the whole argument.
\begin{proposition}\label{3.9}
Consider the map $H^1(\mathcal{K}_v,\mathbf{T})/H_{\mathcal{F}_v}^1(\mathcal{K}_v,\mathbf{T})\rightarrow H_{\mathcal{F}^*}^1(\mathcal{K},\mathbf{A})^*$ which is the Pontryagin dual of the natural map
$$H_{\mathcal{F}^*}^1(\mathcal{K},\mathbf{A})\rightarrow H_{\mathcal{F}_v}^1(\mathcal{K}_v,\mathbf{A}).$$
Then for any height one prime $P$ of $\Lambda$ as above. We localize the above map at $P$ and compose it with projection to the free-$\Lambda_P$ part of $H_{\mathcal{F}^*}^1(\mathcal{K},\mathbf{A})^*_P$. We write this $\Lambda_P$-module map as $j_{P,v}$. Then:
$$\mathrm{ord}_P(\mathrm{coker}j_{P^\iota,v})=f_{v,P}.$$
We can define $j_{P,\bar{v}}$ similarly and have $\mathrm{ord}_P(\mathrm{coker}j_{P^\iota,\bar{v}})=f_{\bar{v},P}$.
\end{proposition}
The proposition follows from the following lemma:
\begin{lemma}
Let $v$ be a prime of $\mathcal{K}$ above $p$. Then
$$\mathrm{ord}_P\mathrm{coker}\{H_\mathcal{F}^1(\mathcal{K},\mathbf{T})_P\rightarrow H_\mathcal{F}^1(\mathcal{K}_v,\mathbf{T})_P\}=\mathrm{ord}_{P^{\iota}}\mathrm{coker}\{H_\mathcal{F}^1(\mathcal{K}_v,\mathbf{A})^*_{P^{\iota}}\rightarrow X_{P^{\iota}}\rightarrow \Lambda_{P^{\iota}}\}$$
where the last arrow is defined as follows. Recall Howard proved the quasi-isomorphism: $X\rightarrow \Lambda\oplus M\oplus M$ with finite kernel and cokernel. That arrow is induced from composing this map with projection to $\Lambda$.
\end{lemma}
\begin{proof}
We claim that $\mathrm{length}_{\mathbb{Z}_p}H^0(\mathcal{K}_v,\mathbf{A})$ is bounded by a constant depending only on $\mathbf{T}$. Take $\gamma\in I_v$ such that $\epsilon(\gamma)\not\equiv 1(\mathrm{mod}\mathfrak{m})$ take a basis
$v_1,v_2$ of $T$ such that the action of $\gamma$ is diagonal $\begin{pmatrix}a_\gamma&\\&d_\gamma\end{pmatrix}$. Suppose $d_\gamma\not\equiv 1(\mathrm{mod}\mathfrak{m})$. Then if $v\in H^0(\mathcal{K}_v,\mathbf{A})$, then $v=a_1v_1$ for some $a_1\in\Lambda^*$. Moreover by considering the action of $I_v$ we get: $Wa_1=0$. Also $(a_p-1)a_1=0$ by considering the action of an element in $H$ which induces the Frobenius modulo $I_v$. The claim is clear.

We follow the idea of \cite[Theorem 2.2.10]{Howard}. For any height one prime P, if $P$ is not $(p)$ we take a generator $g$ of it. Let $f=g+p^m$ and $\mathfrak{D}$ be the ideal generated by $f$. Suppose $\mathfrak{D}$ is not in $\Sigma_\Lambda$ and outside the support of $M$. Then if $m$ is large then we have $\Lambda/P\simeq \Lambda/\mathfrak{D}$ as rings (not as $\Lambda$-modules). Thus $\mathfrak{D}$ is a height one prime as well. In the following we use $\approx$ to mean the difference is a contant not depending on $\mathfrak{D}$ and $m$. We write LHS and RHS for the left hand side and right hand side of the equality in the proposition. Then for $p=v\bar{v}$
\begin{align*}
&&&LHS\times md&\\
&\approx&&\mathrm{length}_{\mathbb{Z}_p}\mathrm{coker}\{H_\mathcal{F}^1(\mathcal{K},\mathbf{T})/\mathfrak{D}\rightarrow H_\mathcal{F}^1(\mathcal{K}_v,\mathbf{T})/\mathfrak{D}\}&\\
&\approx&&\mathrm{length}_{\mathbb{Z}_p}\mathrm{coker}\{H_\mathcal{F}^1(\mathcal{K},\mathbf{T}_\mathfrak{D})\rightarrow H_\mathcal{F}^1(\mathcal{K}_v,\mathbf{T}_\mathfrak{D})\}&
\end{align*}
Here we used \cite[lemma 2.2.7,lemma 2.2.8]{Howard}. We also used that $H^1(\mathcal{K}_v,\mathbf{T})/\mathfrak{D}\hookrightarrow H^1(\mathcal{K}_v,\mathbf{T}/\mathfrak{D})\rightarrow H^2(\mathcal{K}_v,\mathbf{T})[\mathfrak{D}]$, and the last term is finite whose size is independent of $\mathfrak{D}$ and $m$ by the claim above and Tate duality.
\begin{align*}
&&&RHS\times md&\\
&\approx&&\mathrm{length}_{\mathbb{Z}_p}\mathrm{coker}\{H_\mathcal{F}^1(\mathcal{K}_v,\mathbf{A})^*/\mathfrak{D}^{\iota}\rightarrow X/\mathfrak{D}^{\iota}\rightarrow \Lambda/\mathfrak{D}^{\iota}\}&\\
&\approx&&\mathrm{length}_{\mathbb{Z}_p}\mathrm{ker}\{\Lambda^*[\mathfrak{D}^{\iota}]\rightarrow H_\mathcal{F}^1(\mathcal{K},\mathbf{A})[\mathfrak{D}^{\iota}]\rightarrow H_\mathcal{F}^1(\mathcal{K}_v,\mathbf{A})[\mathfrak{D}^{\iota}]\}&\\
&\approx&&\mathrm{length}_{\mathbb{Z}_p}\mathrm{ker}\{H_\mathcal{F}^1(\mathcal{K},\mathbf{A})[\mathfrak{D}^{\iota}]_{div}\rightarrow H_\mathcal{F}(\mathcal{K}_v,\mathbf{A})[\mathfrak{D}^{\iota}]\}&\\
&\approx&&\mathrm{length}_{\mathbb{Z}_p}\mathrm{ker}\{H_\mathcal{F}^1(\mathcal{K},\mathbf{A}_{\mathfrak{D}^{\iota}})_{div}\rightarrow H_\mathcal{F}^1(\mathcal{K}_v,\mathbf{A}_{\mathfrak{D}^{\iota}})\}&
\end{align*}
Here the subscript $div$ means the divisible part and again we used \cite[lemma 2.2.7, lemma 2.2.8]{Howard}. Moreover we need that the kernel of the map $H_\mathcal{F}^1(\mathcal{K}_v,\mathbf{A}[\mathfrak{D}^{\iota}])\rightarrow H_\mathcal{F}^1(\mathcal{K}_v,\mathbf{A})[\mathfrak{D}^{\iota}]$ is bounded by some constant depending only on $\mathbf{T}$.

Note that $\mathbf{T}_\mathfrak{D}$ is the $\mathfrak{m}$-adic Tate module of $\mathbf{A}_{\mathfrak{D}^{\iota}}$. We have $H_\mathcal{F}^1(\mathcal{K},\mathbf{T}_\mathfrak{D})=\varprojlim_i H_\mathcal{F}^1(\mathcal{K},\mathbf{T}_\mathfrak{D}/\mathfrak{m}^i)$ and $H_\mathcal{F}^1(\mathcal{K},\mathbf{A}_{\mathfrak{D}^{\iota}})=\varinjlim_i H_\mathcal{F}^1(\mathcal{K},\mathfrak{m}^{-i}\mathbf{T}_\mathfrak{D}/\mathbf{T}_\mathfrak{D})$ (by \cite[lemma 1.3.3]{Howard}). The similar identities are true for the local cohomology groups. So the lemma follows by taking $m$ tends to infinity.

If $P$ is $(p)$, we take the height one prime $\mathfrak{D}=(p+T^m)$ and argue similarly. Although we do not have the isomorphism $\Lambda/P\simeq \Lambda/\mathfrak{D}$, we do have $S_{\mathfrak{D}}=\Lambda/\mathfrak{D}$. See \cite[Theorem 2.2.10]{Howard}.
\end{proof}

\begin{proposition}\label{3.11}
For a height one prime $P$ of $\Lambda$, we have:
$H_{\mathcal{F}_v}^1(\mathcal{K},\mathbf{A})^*$ is $\Lambda$-torsion module and:
$$2\mathrm{ord}_PM+f_{\bar{v},P^\iota}=\mathrm{ord}_P(H_{\mathcal{F}_v}^1(\mathcal{K},\mathbf{A})^*).$$
\end{proposition}
\begin{proof}
In the poitou-Tate exact sequence we take $\mathcal{G}^*=\mathcal{F}_v$ and $\mathcal{F}^*$ to be our $\mathcal{F}$. The proposition follows from lemma \ref{3.4} and proposition \ref{3.9}.
\end{proof}
\begin{proposition}\label{3.12}
For a height one prime $P$ as above we have:
$H_v^1(\mathcal{K},\mathbf{A})^*$ is $\Lambda$-torsion and:
$$2\mathrm{ord}_P(M)+f_{\bar{v},P^\iota}+f_{v,P}=\mathrm{ord}_P(H_v^1(\mathcal{K},\mathbf{A})^*).$$
\end{proposition}
\begin{proof}
In the Poitou-Tate exact sequence we take $\mathcal{F}^*$ to be $v$ and $\mathcal{G}^*=\mathcal{F}_v$. The proposition follows from the last proposition, lemma \ref{3.4} and lemma \ref{3.5}.
\end{proof}
Note that we have $f_{\bar{v},P^\iota}=f_{v,P}$ by considering the complex conjugation of the Galois representation. Thus we in fact proved:
\begin{equation}\label{(1)}
2\mathrm{ord}_P(M)+2f_{v,P}=\mathrm{ord}_P(H_v^1(\mathcal{K},\mathbf{A})^*)
\end{equation}

Before proving the main theorem, let us prove the following corollary:
\begin{corollary}
The $2$-variable Selmer module $X_{f,\mathcal{K}}$ defined in \cite{WAN} is $\Lambda$-torsion.
\end{corollary}
\begin{proof}
Suppose it is not. Then by the control theorem for Selmer groups the anti-cyclotomic Selmer module with ``$v$'' Selmer condition at $p$ is not torsion, which contradicts the above proposition.
\end{proof}
\begin{theorem}
Theorem \ref{1.2} is true, i.e. under assumption (2) we have
$$\mathrm{char}_\Lambda(M)=\mathrm{char}_\Lambda(H_\mathcal{F}^1(\mathcal{K},\mathbf{T})/\Lambda\kappa_1^{Hg}).$$
Under assumption (1) this is true as ideals of $\Lambda\otimes_{\mathbb{Z}_p}\mathbb{Q}_p$.
\end{theorem}
\begin{proof}
Fix a height one prime $P$. By (\ref{(1)}), corollary \ref{Castella} and corollary \ref{3.8} we have:
$$2\mathrm{ord}_PM+2f_{v,P}\geq 2\mathrm{ord}_P(H_\mathcal{F}^1(\mathcal{K},\mathbf{T})/\kappa_1^{Hg})+2f_{v,P}.$$
The theorem follows.
\end{proof}
Now let us prove the two variable main conjecture.
\begin{theorem}\label{main conjecture}
Suppose assumption (2) is true. Then $\mathcal{L}_{f,\mathcal{K}}\in\Lambda_\mathcal{K}$. Moreover:
$$\mathrm{char}_{\Lambda_\mathcal{K}}(X_{f,\mathcal{K}})=(\mathcal{L}_{f,\mathcal{K}}).$$
\end{theorem}
\begin{remark}
After I finished the first draft B. Howard informed me that Francesc Castella obtained similar results in \cite{Castella} via a seemingly different proof.
\end{remark}
\begin{proof}
Let $\tilde{\mathcal{L}}_{f,\mathcal{K}}$ be the $\mathcal{L}_{f,\mathcal{K}}$ multiplied by its denominator (note that $\Lambda_\mathcal{K}$ is uniquely factorable ring) and then remove all $p$-power divisors (so $\tilde{\mathcal{L}}_{f,\mathcal{K}}\in\Lambda_\mathcal{K}$). By the main theorem of \cite{WAN} we have the two variable
\begin{equation}\label{2}
(\tilde{\mathcal{L}}_{f,\mathcal{K}})\supseteq \mathrm{char}_{\Lambda_\mathcal{K}}(X_{f,\mathcal{K}})
\end{equation}
(note that $\mathrm{ord}_p(\tilde{\mathcal{L}}_{f,\mathcal{K}})=0$).

Write $\tilde{\mathcal{L}}_{f,\mathcal{K}}^{anti}$ for $\tilde{\mathcal{L}}_{f,\mathcal{K}}$ evaluated at $\gamma^+=1$. By \cite{Hsieh} $\mathcal{L}^{BDP}_{f,\mathcal{K}}$ has $\mu$-invariant $0$. So it is not hard to see from the construction that $\mathcal{L}_{f,\mathcal{K}}^{BDP}|\tilde{\mathcal{L}}_{f,\mathcal{K}}^{anti}$. Then by Theorem \ref{How}, corollary \ref{Castella}, and (\ref{(1)}) we already have:
\begin{equation}\label{3}
\mathrm{char}_\Lambda(X_{f,\mathcal{K}}^{anti})\supseteq (\mathcal{L}_{f,\mathcal{K}}^{BDP})\supseteq (\tilde{\mathcal{L}}_{f,\mathcal{K}}^{anti})
\end{equation}
as ideals of $\Lambda$.
So we deduce that
$$(\tilde{\mathcal{L}}_{f,\mathcal{K}})=\mathrm{char}_{\Lambda_\mathcal{K}}(X_{f,\mathcal{K}})$$
from (\ref{2}), (\ref{3}) in the same way as \cite[Theorem 3.6.5]{SU} and all above ``$\supseteq$'' are $=$.  In particular $(\mathcal{L}_{f,\mathcal{K}}^{BDP})=(\tilde{\mathcal{L}}_{f,\mathcal{K}}^{anti})$. Now we do some study on the possible denominator of $\mathcal{L}_{f,\mathcal{K}}$. There is another construction of a similar $p$-adic $L$-function $\mathcal{L}^{Hida}_{f,\mathcal{K}}\in\Lambda_\mathcal{K}\otimes_\Lambda\mathrm{Frac}\Lambda$ in \cite{Hida} using Rankin-Selberg method for $f$ and the Hida family $\mathbf{g}$ of normalized CM eigenforms associated to characters of $\Gamma_\mathcal{K}$, and such that the weight of $f$ is lower than the specializations of $\mathbf{g}$ at interpolation points. The interpolation formula is almost the same as $\mathcal{L}_{f,\mathcal{K}}$ except that the period is given by the Petersson inner product of specializations of $\mathbf{g}$ instead of the CM period $\Omega_\infty$. The ratio of the two periods is given by a Katz $p$-adic $L$-function which is an element in $\Lambda\backslash \{0\}$. Comparing these two $p$-adic $L$-functions we see that the denominator for $\mathcal{L}_{f,\mathcal{K}}$ must be generated by an element $E$ in $\Lambda\backslash\{0\}$ (in other word it involves the anticyclotomic variable only). Recall that $\tilde{\mathcal{L}}_{f,\mathcal{K}}=\frac{\mathcal{L}_{f,\mathcal{K}}\cdot E}{p^m}$ for some $m\geq 0$ and note $\mathcal{L}_{f,\mathcal{K}}|_{\gamma^+=1}=\mathcal{L}^{BDP}_{f,\mathcal{K}}$, then  $(\mathcal{L}_{f,\mathcal{K}}^{BDP})=(\tilde{\mathcal{L}}_{f,\mathcal{K}}^{anti})$ implies $E=p^m$ up to a unit in $\Lambda$. So $(\tilde{\mathcal{L}}_{f,\mathcal{K}})=(\mathcal{L}_{f,\mathcal{K}})$. This proves the theorem.
\end{proof}
Finally we prove the following strengthening of the result of \cite{Skinner}.
\begin{theorem}
Let $E/\mathbb{Q}$ be an elliptic curve with square-free conductor $N$ and good ordinary reduction at $p$. Let $\mathcal{K}$ be an imaginary quadratic field such that $p$ and $2$ are split and such that $E[p]|_{G_\mathcal{K}}$ is irreducible. Suppose moreover assumption (1) is true. If the Selmer group $H_{\mathcal{F}}^1(\mathcal{K}, E[p^\infty])$ has corank one. Then the Heegner point $\kappa_1$ is not torsion and thus the vanishing order of $L_\mathcal{K}(f,1)$ is exactly one.
\end{theorem}
\begin{proof}
We consider the behavior of specialization of the Selmer group $H_{\mathcal{F}}^1$ from $\mathcal{K}_\infty^-$ to $\mathcal{K}$. By our discussion on Greenberg's method in subsection \ref{control} we know that the surjection
$$ X_{f,\mathcal{K}}/IX_{f,\mathcal{K}}\twoheadrightarrow H_\mathcal{F}^1(\mathcal{K}, E[p^\infty])^*$$
has finite kernel. By theorem 1.1 and theorem 1.2 we have $M\otimes\Lambda/I$ and thus $\frac{H_\mathcal{F}^1(\mathcal{K}, \mathbf{T})}{\Lambda\kappa_1^{Hg}}\otimes \frac{\Lambda}{I}$ is torsion. Thus $\frac{H_\mathcal{F}^1(\mathcal{K},\mathbf{T})}{IH_\mathcal{F}^1(\mathcal{K},\mathbf{T})+\Lambda\kappa_1^{Hg}}$ is torsion. One can easily check that there is an injection
$$H_\mathcal{F}^1(\mathcal{K},\mathbf{T}\otimes_{\mathbb{Z}_p}\mathbb{Q}_p)/IH_\mathcal{F}^1(\mathcal{K},\mathbf{T}\otimes_{\mathbb{Z}_p}\mathbb{Q}_p)\hookrightarrow H_\mathcal{F}^1(\mathcal{K},T\otimes_{\mathbb{Z}_p}\mathbb{Q}_p).$$
In sum the image $\kappa_1$ of $\kappa_1^{Hg}$ is not torsion in $H^1_\mathcal{F}(\mathcal{K},T\otimes_{\mathbb{Z}_p}\mathbb{Q}_p)$. So $\kappa_1$ is not torsion.
\end{proof}

\textsc{Xin Wan, Department of Mathematics, Columbia University, New York, NY, 10027, USA}\\
\indent \textit{E-mail Address}: xw2295@math.columbia.edu\\

\end{document}